\documentclass[12pt,reqno]{amsart}
\usepackage{hyperref}
\usepackage{url}
\usepackage{amssymb}
\usepackage{tikz,enumerate}
\usepackage{diagbox}
\usepackage{appendix}
\allowdisplaybreaks[4]

\newtheorem{theorem}{Theorem}

\newtheorem{conj}{Conjecture}

\theoremstyle{definition}

\theoremstyle{remark}

\numberwithin{equation}{section}

\begin{document}
\title[Farey sequence and Graham's conjectures]
 {Farey sequence and Graham's conjectures}

\author{Liuquan Wang}
\address{School of Mathematics and Statistics, Wuhan University, Wuhan 430072, Hubei, People's Republic of China}
\email{wanglq@whu.edu.cn;mathlqwang@163.com}



\begin{abstract}
Let ${F}_{n}$ be the Farey sequence of order $n$. For $S \subseteq {F}_n$ we let $\mathcal{Q}(S) = \left\{x/y:x,y\in S, x\le y \, \, \textrm{and} \, \,  y\neq 0\right\}$. We show that if $\mathcal{Q}(S)\subseteq F_n$, then $|S|\leq n+1$. Moreover, we prove that in any of the following cases: (1) $\mathcal{Q}(S)=F_n$; (2) $\mathcal{Q}(S)\subseteq F_n$ and $|S|=n+1$, we must have $S = \left\{0,1,\frac{1}{2},\ldots,\frac{1}{n}\right\}$ or $S=\left\{0,1,\frac{1}{n},\ldots,\frac{n-1}{n}\right\}$ except for $n=4$, where we have an additional set $\{0, 1, \frac{1}{2}, \frac{1}{3}, \frac{2}{3}\}$ for the second case. Our results are based on Graham's GCD conjectures, which have been proved by Balasubramanian and  Soundararajan.
\end{abstract}

\maketitle

\section{Introduction and Main Results}
Given a positive integer $n$, the Farey sequence $F_n$ is the set of rational numbers $a/b$ with $0\leq a \leq b \leq n$ and $(a,b)=1$. Throughout this paper, we use the notation $(a,b)$ and $[a,b]$ to denote the greatest common divisor and least common multiple of $a$ and $b$, respectively. The Farey sequence is arranged in increasing order. For instance,
$$F_5=\left\{\frac{0}{1},\frac{1}{5},\frac{1}{4},\frac{1}{3},\frac{2}{5},\frac{2}{3},\frac{3}{4},\frac{4}{5},\frac{1}{1}\right\}.$$

In this paper, we present some interesting properties of Farey sequences. For any set $S$ of real numbers, we define
\[\mathcal{Q}(S) = \left\{\frac{x}{y}:x,y\in S,  x\le y  \, \, \textrm{and} \, \, y\neq 0\right\}.\]
In particular, if $S=\{0\}$ then we agree that $\mathcal{Q}(S)=\{0\}$.
 We are mainly interested in the following  question: can we find all subsets $S \subseteq {F}_n$, such that $\mathcal{Q}(S)={F}_{n}$?
It turns out that this question is closely related to Graham's greatest common divisor problem. In 1970, Graham \cite{Graham} proposed the following conjecture.
\begin{conj}\label{Gthm}
Let ${{a}_{1}},{{a}_{2}},\cdots ,{{a}_{n}}$ be distinct positive integers, we have
\[\underset{i,j}{\mathop{\max }}\,\frac{{{a}_{i}}}{({{a}_{i}},{{a}_{j}})}\ge n.\]
\end{conj}
Let $M_n=\mathrm{l.c.m.}\{1,2,\dots,n\}$. It was pointed out in \cite{Velez} that Graham also made the following stronger conjecture.
\begin{conj}\label{conj2}
Assume $a_1<a_2<\cdots <a_n$,
\[\mathrm{g.c.d.}\{a_1,a_2,\dots, a_n\}=1 \quad \textrm{and} \quad \underset{i,j}{\mathop{\max }}\frac{a_i}{(a_i,a_j)}=n.\]
Then $\{a_1,a_2,\dots, a_n\}$ can only be $\{1,2,\dots, n\}$ or $\{\frac{M_n}{n},\frac{M_n}{n-1},\dots, \frac{M_n}{1}\}$ except for $n=4$, where we have the additional sequence $\{2,3,4,6\}$.
\end{conj}
Graham's conjectures have stimulated a number of researches. For instance, Winterle \cite{Winterlez} established Graham's first conjecture in the case when there is a prime among $a_i$'s. V\'{e}lez \cite{Velez} showed that Conjecture \ref{Gthm} is true for $n=p+1$ ($p$ a prime). A proof for $n=p$ was also included in \cite{Velez}, which is due to Szemer\'{e}di. Boyle \cite{Boyle} showed that Conjecture \ref{Gthm} holds when there is a prime $p>(n-1)/2$ dividing some $a_i$. Szegedy \cite{Szegedy} and Zaharescu \cite{Za} independently showed that Graham's conjectures hold for sufficiently large $n$. Cheng and Pomerance \cite{Cheng-P} showed that Graham's second conjecture is true for $n>10^{50000}$. It was Balasubramanian and Soundararajan \cite{BS} who finally confirmed these conjectures based on deep analytical methods.

From Conjecture \ref{Gthm} we get the following result.
\begin{theorem}\label{thm1}
Suppose $S\subseteq {{F}_{n}},$ if $\mathcal{Q}(S)\subseteq {{F}_{n}}$, then $S$ has at most $n+1$ elements.
\end{theorem}
In fact, the above theorem is equivalent to Conjecture \ref{Gthm}.
\begin{theorem}\label{equiv}
Conjecture \ref{Gthm} is equivalent to Theorem \ref{thm1}.
\end{theorem}

Now we give the answer to our  question proposed before.
\begin{theorem}\label{answer}
Suppose $S \subseteq {F}_n$ and $\mathcal{Q}(S)={F}_{n}$, then $S$ can only be one of the following two sets:
\begin{displaymath}
S = \left\{0,1,\frac{1}{2},\ldots,\frac{1}{n}\right\} \quad
\textrm{ or } \quad S = \left\{0,1,\frac{1}{n},\ldots,\frac{n-1}{n}\right\}.
\end{displaymath}
\end{theorem}

We also have the following result based on Graham's second conjecture.
\begin{theorem}\label{thm-4}
 Suppose $S \subseteq {F}_n$, $|S|=n+1$ and $\mathcal{Q}(S)\subseteq {F}_{n}$, then $S=\left\{0,1,\frac{1}{2},\ldots,\frac{1}{n}\right\}$ or $S=\left\{0,1,\frac{1}{n},\ldots,\frac{n-1}{n}\right\}$ except for $n=4$, where we have an additional set $\{0,1,\frac{1}{2}, \frac{1}{3}, \frac{2}{3}\}$.
\end{theorem}
In a way similar to the proof of Theorem \ref{equiv}, it is not difficult to show that Theorem \ref{thm-4} also implies Conjecture \ref{conj2}. Since the proof of Graham's conjectures  in \cite{BS} is relatively long and complicate, it is natural to ask whether we can prove Theorems \ref{thm1} and \ref{thm-4} directly and thus providing new proofs for Graham's conjectures.
\section{Proofs of the Theorems}
\begin{proof}[Proof of Theorem \ref{thm1}]
We prove by contradiction. Suppose there is a subset $S\subseteq \mathcal{F}_{n}$ such that $\mathcal{Q}(S) \subseteq \mathcal{F}_{n}$ but $|S|\ge n+2$. Then $S\backslash\{0,1\}$ has at least $n$ distinct elements $x_k/y_k$ with $({{x}_{k}},{{y}_{k}})=1$ ($1\leq k \leq n$) and
\begin{align*}
\frac{x_1}{y_1}<\frac{x_2}{y_2}<\cdots <\frac{x_n}{y_n}.
\end{align*}
Let
\[S'=\left\{\frac{x_1}{y_1},\ldots,\frac{x_{n}}{y_{n}},1\right\}.\]
It is easy to see that $\mathcal{Q}(S') \subseteq {{F}_{n}}$.

Let $x_{n+1}=y_{n+1}=1$, ${{a}_{n+1}}=[{{x}_{1}},\dots ,{{x}_{n}}]$, ${{a}_{k}}={{a}_{n+1}}\cdot
\frac{{{y}_{k}}}{{{x}_{k}}}$ ($1\le k\le n$). Then
\[a_1>a_2>\cdots>a_{n+1}.\]
Moreover, we have for $1\le j\le n$,
\[(a_{n+1},a_{j})=(\frac{a_{n+1}}{x_j}{x_j}, \frac{a_{n+1}}{x_{j}}{y_j})=\frac{a_{n+1}}{x_j},  \]
and for $1\leq i <j\leq n$,
\begin{align*}
({{a}_{i}},{{a}_{j}}) & =({{a}_{n+1}}\cdot \frac{{{y}_{i}}}{{{x}_{i}}},
{{a}_{n+1}}\cdot \frac{{{y}_{j}}}{{{x}_{j}}})\\
&=\frac{{{a}_{n+1}}}{[{{x}_{i}},{{x}_{j}}]}\cdot ({{y}_{i}},{{y}_{j}})\left(\frac{[{{x}_{i}},{{x}_{j}}]}{{{x}_{i}}}\cdot \frac{{{y}_{i}}}{({{y}_{i}},{{y}_{j}})},\frac{[{{x}_{i}},{{x}_{j}}]}{{{x}_{j}}}\cdot \frac{{{y}_{j}}}{({{y}_{i}},{{y}_{j}})}\right).
\end{align*}
Note that
$\frac{[{{x}_{i}},{{x}_{j}}]}{{{x}_{i}}}|{{x}_{j}},\frac{{{y}_{j}}}{({{y}_{i}},{{y}_{j}})}|{{y}_{j}},({{x}_{j}},{{y}_{j}})=1$, we have
$(\frac{[{{x}_{i}},{{x}_{j}}]}{{{x}_{i}}},\frac{{{y}_{j}}}{({{y}_{i}},{{y}_{j}})})=1$.  Similarly we have $(\frac{{{y}_{_{i}}}}{({{y}_{i}},{{y}_{j}})},\frac{[{{x}_{i}},{{x}_{j}}]}{{{x}_{j}}})=1$.  Hence for $1\leq i <j \leq n+1$,
\begin{align}
({{a}_{i}},{{a}_{j}})=\frac{{{a}_{n+1}}}{[{{x}_{i}},{{x}_{j}}]}({{y}_{i}},{{y}_{j}})
\end{align}
and
\begin{align}\label{gcd-eq}
\frac{{{a}_{i}}}{({{a}_{i}},{{a}_{j}})}=\frac{{{a}_{i}}}{{{a}_{n+1}}}\cdot \frac{[{{x}_{i}},{{x}_{j}}]}{({{y}_{i}},{{y}_{j}})}=\frac{{{y}_{i}}}{{{x}_{i}}}\cdot \frac{[{{x}_{i}},{{x}_{j}}]}{({{y}_{i}},{{y}_{j}})}.
\end{align}
On the other hand,
\[\frac{{{x}_{i}}}{{{y}_{i}}}/\frac{{{x}_{j}}}{{{y}_{j}}}=\frac{{{x}_{i}}{{y}_{j}}}{{{x}_{j}}{{y}_{i}}}=\frac{{{x}_{i}}}{({{x}_{i}},{{x}_{j}})}\cdot \frac{{{y}_{j}}}{({{y}_{i}},{{y}_{j}})}/\left(\frac{{{x}_{j}}}{({{x}_{i}},{{x}_{j}})}\cdot \frac{{{y}_{i}}}{({{y}_{i}},{{y}_{j}})}\right),\]
where the fraction in the right side is in its lowest term. Since $\mathcal{Q}(S')\subseteq {F}_{n}$, we have
\[\frac{{{x}_{j}}}{({{x}_{i}},{{x}_{j}})}\cdot \frac{{{y}_{i}}}{({{y}_{i}},{{y}_{j}})}\le n.\]
This implies
\[\frac{[{{x}_{i}},{{x}_{j}}]}{x_{i}}\cdot \frac{{{y}_{i}}}{({{y}_{i}},{{y}_{j}})}\le n.\]
Utilizing this in \eqref{gcd-eq}, we deduce that for $1\leq i <j \leq n+1$,
\begin{align*}
 \frac{{{a}_{i}}}{({{a}_{i}},{{a}_{j}})}\le n.
\end{align*}
This contradicts with Conjecture \ref{Gthm}.
\end{proof}

\begin{proof}[Proof of Theorem \ref{equiv}]
We have already shown that Conjecture \ref{Gthm} implies Theorem \ref{thm1}.

Now we assume Theorem \ref{thm1} and give a  proof of Conjecture \ref{Gthm}.

Conjecture 1 is obviously true when $n=1$. For $n\ge 2$, we prove by contradiction. Suppose for some integer $n+1$ ($n\ge 1$), Conjecture 1 does not hold. Then there are $n+1$ distinct positive integers ${{a}_{1}}<{{a}_{2}}<\cdots <{{a}_{n+1}}$ such that
\begin{align}\label{gcd-cond}
\underset{i,j}{\mathop{\max }}\,\frac{{{a}_{i}}}{({{a}_{i}},{{a}_{j}})}\le n.
\end{align}
Let ${{x}_{k}}=\frac{{{a}_{1}}}{({{a}_{1}},{{a}_{k}})}$ and ${{y}_{k}}=\frac{{{a}_{k}}}{({{a}_{1}},{{a}_{k}})}$ ($1\le k\le n+1$). Then $({{x}_{k}},{{y}_{k}})=1$, ${{x}_{k}}\leq {{y}_{k}}\le n$, and $\frac{{{x}_{k}}}{{{y}_{k}}}=\frac{{{a}_{1}}}{{{a}_{k}}}$ ($1\le k\le n+1$) are distinct reduced fractions. Thus $S=\{0,\frac{{{x}_{1}}}{{{y}_{1}}},\frac{{{x}_{2}}}{{{y}_{2}}},\cdots ,\frac{{{x}_{n+1}}}{{{y}_{n+1}}}\}$ is a subset of $F_n$. By \eqref{gcd-cond} we have
\[\frac{{{x}_{j}}}{{{y}_{j}}}/\frac{{{x}_{i}}}{{{y}_{i}}}=\frac{{{a}_{1}}/{{a}_{j}}}{{{a}_{1}}/{{a}_{i}}}=\frac{{{a}_{i}}}{{{a}_{j}}}=\frac{{{a}_{i}}/({{a}_{i}},{{a}_{j}})}{{{a}_{j}}/({{a}_{i}},{{a}_{j}})}\in F_{n}.\]
Hence $\mathcal{Q}(S)\subseteq F_{n}$. But $|S|\ge n+2$, which contradicts with Theorem \ref{thm1}. Therefore, Conjecture \ref{Gthm} follows from Theorem \ref{thm1}.
\end{proof}

\begin{proof}[Proof of Theorem \ref{answer}]
It is easy to check that $\mathcal{Q}(S)={{F}_{n}}$ when $S=\{0,1,\frac{1}{2},\cdots ,\frac{1}{n}\}$  or $S=\{0,1,\frac{1}{n},\frac{2}{n},\cdots ,\frac{n-1}{n}\}$.
Now we are going to show that there are no other such subsets.

When $n\le 4$, the theorem can be proved by direct verification.

When $n\ge 5$, since $0\in {{F}_{n}}$, we must have $0\in S$.  By  Bertrand's Postulate, there is  a prime $p$ satisfying $\frac{n+1}{2}\leq p< n$. Let ${{\overline{{F}}}_{n}}$ denote the set of all nonnegative fractions with both denominator and numerator no bigger than $n$. Since $1/p\in F_n$ and $\mathcal{Q}(S)={{F}_{n}}$, there are two reduced fractions $\frac{{{x}_{1}}}{{{y}_{1}}},\frac{{{x}_{2}}}{{{y}_{2}}}\in S$ such that $\frac{{{x}_{1}}}{{{y}_{1}}}/\frac{{{x}_{2}}}{{{y}_{2}}}=\frac{1}{p}$, i.e., $p{{x}_{1}}{{y}_{2}}={{x}_{2}}{{y}_{1}}$. Thus either $x_2$ or $y_1$ is divisible by $p$. Since $2p\ge n+1$, we must have ${{x}_{2}}=p$ or ${{y}_{1}}=p$.

{\bf Case 1}:  ${{x}_{2}}=p$. In this case, $x_1y_2=y_1$. Since $({{x}_{1}},{{y}_{1}})=1$, we  have ${{x}_{1}}=1$ and $y_1=y_2$. Hence there is an integer $c=y_1=y_2$ ($p<c\le n$) such that $\frac{p}{c}\in S$ and $\frac{1}{c}\in S$.  Suppose $\frac{p}{x}\in S$ ($p<x\le n$), then for any reduced fraction $\frac{a}{b}\in S$, we have $\frac{a}{b}/\frac{p}{x}=\frac{ax}{bp}\in {\overline{F}_{n}}$. Because $(a,b)=(x,p)=1$, we know that $p=a$ or $b|x$. In particular, $\frac{1}{c}\in S$ implies $c|x$. But $2c>2p\ge n+1$, we must have $x=c$. Therefore, $S$ can be written as
\[S=\left\{\frac{{{a}_{1}}}{{{b}_{1}}},\frac{{{a}_{2}}}{{{b}_{2}}},\cdots ,\frac{{{a}_{s}}}{{{b}_{s}}}\right\}\bigcup \left\{\frac{p}{c}, 0\right\},\]
where $({{a}_{i}},{{b}_{i}})=1$, $({{a}_{i}},p)=1$, $b_i|c$ ($1\le i\le s$).

Note that $\frac{{{a}_{i}}}{{{b}_{i}}}/\frac{{{a}_{j}}}{{{b}_{j}}}=\frac{{{a}_{i}}{{b}_{j}}}{{{a}_{j}}{{b}_{i}}}$. We see that both the numerator and denominator in the lowest term of this fraction are not divisible by $p$. Since $\frac{x}{p}\in {{\overline{{F}}}_{n}}$ ($1\le x\le n,x\ne p$), it can only be represented as $\frac{{{a}_{i}}}{{{b}_{i}}}/\frac{p}{c}=\frac{{{a}_{i}}c}{p{{b}_{i}}}$. However, the numerator  of $\frac{{{a}_{i}}c}{p{{b}_{i}}}$  is $\frac{{{a}_{i}}c}{{{b}_{i}}}\leq c$. Since $\frac{n}{p}\in {{\overline{{F}}}_{n}}$, we must have $c=n$.

For any $1\le x\le n$, $x\ne p$, there is some $\frac{{{a}_{i}}}{{{b}_{i}}}$ such that $\frac{x}{p}=\frac{{{a}_{i}}}{{{b}_{i}}}/\frac{p}{n}=\frac{n{{a}_{i}}}{p{{b}_{i}}}$, i.e., $\frac{{{a}_{i}}}{{{b}_{i}}}=\frac{x}{n}$. Therefore,
we know that $\{0\}\cup {{\{\frac{k}{n}\}}_{1\le k\le n}}\subseteq S$. By Theorem \ref{thm1} we have $|S|\le n+1$. Hence $S={{\{\frac{k}{n}\}}_{0\le k\le n}}$.

{\bf Case 2}:  ${{y}_{1}}=p$. In this case, ${{x}_{1}}{{y}_{2}}={{x}_{2}}$. Since $({{x}_{2}},{{y}_{2}})=1$, we must have $y_2=1$. It follows that $x_1=x_2=1$.  For any reduced fraction $a/b\in S$, we have $\frac{a}{b}/\frac{1}{p}=\frac{ap}{b}\in \overline{F}_n$. We must have $b=p$ or $a=1$. Thus we can write
\[
S=\left\{0, 1,\frac{1}{{{p}}}\right\}\bigcup \left\{\frac{{{a}_{1}}}{{{p}}},\frac{{{a}_{2}}}{{{p}}},\cdots ,\frac{{{a}_{t}}}{{{p}}}\right\}\bigcup \left\{\frac{1}{{{b}_{1}}},\frac{1}{{{b}_{2}}},\cdots ,\frac{1}{{{b}_{r}}}\right\}.
\]
Here $1<{{a}_{1}}<{{a}_{2}}<\cdots <{{a}_{t}}<p$ and $1<{{b}_{1}}<{{b}_{2}}<\cdots <{{b}_{r}}\leq n$ and $b_i\neq p$ for all $1\leq i \leq r$.

The smallest positive number in $F_n$ is $\frac{1}{n}$. Since $p<n$, the smallest positive number in $\mathcal{Q}(S)$ should be $\frac{1}{b_r}$ and $b_r=n$. If $t\geq 1$, then $\frac{a_1b_r}{p}\in \overline{F}_{n}$, i.e., $\frac{a_1n}{p}\in \overline{F}_{n}$, we get ${{a}_{1}}=1$, which is a contradiction. Therefore, $t=0$ and
\[
S=\left\{0, 1,\frac{1}{{{p}}}\right\}\bigcup \left\{\frac{1}{{{b}_{1}}},\frac{1}{{{b}_{2}}},\cdots ,\frac{1}{{{b}_{r}}}\right\}.
\]
For $1\leq k \leq n$, we have $\frac{k}{p}\in \overline{F}_n$. Since $F_n=\mathcal{Q}(S)$, we deduce that $\{1,2,\dots, n-1\}\backslash \{p\}\subseteq \{b_1,b_2,\dots,b_{r-1}\}$. Hence $S=\left\{0,1,\frac{1}{2},\dots, \frac{1}{n}\right\}$.
\end{proof}

%
%

\begin{proof}[Proof of Theorem \ref{thm-4}]
From Theorem \ref{thm1} we know that $\{0,1\}\subseteq S$. Besides 0 and 1, let other elements in $S$ be
\[\frac{x_1}{y_1}<\frac{x_2}{y_2}<\cdots <\frac{x_{n-1}}{y_{n-1}}.\]
Let $x_{n}=y_{n}=1$, ${a}_{n}=\mathrm{l.c.m.}\{x_1,x_2,\dots,x_{n-1}\}$, ${{a}_{k}}={{a}_{n}}\cdot
\frac{{{y}_{k}}}{{{x}_{k}}}(1\le k<n)$. Then
\[a_1>a_2>\cdots>a_{n}.\]
As in the proof of Theorem \ref{thm1}, since $\mathcal{Q}(S)\subseteq F_n$, we see that
\begin{align*}
\underset{i,j}{\mathop{\max }}\,\frac{{{a}_{i}}}{({{a}_{i}},{{a}_{j}})}\le n.
\end{align*}
By Conjecture \ref{Gthm}, we know that
\begin{align*}
\underset{i,j}{\mathop{\max }}\,\frac{{{a}_{i}}}{({{a}_{i}},{{a}_{j}})}= n.
\end{align*}
Therefore, by Conjecture \ref{conj2}, we see that $\{a_1,a_2\cdots, a_{n}\}$ must be one of $\{n,n-1,\dots,1\}$ or $\{\frac{M_{n}}{1},\frac{M_{n}}{2},\dots,\frac{M_{n}}{n}\}$ except for $n=4$, where we have an additional sequence $\{6,4,3,2\}$. The sequence $\{6,4,3,2\}$ corresponds to $S=\{0,1,\frac{1}{2},\frac{1}{3},\frac{2}{3}\}$.  Since $\frac{x_k}{y_k}=\frac{a_n}{a_k}$ ($1\leq k \leq {n-1}$), we get the desired conclusion immediately.
\end{proof}

\subsection*{Acknowledgements}
The author was partially supported by the National Natural Science Foundation of China (11801424) and a start-up research grant of the Wuhan University.

\end{document}